\documentclass[sn-mathphys]{sn-jnl}

\usepackage{subfigure}
\newenvironment{newlist}[1]%
{\begin{list}{}{\settowidth{\labelwidth}{\bf #1}%
			\setlength{\leftmargin}{\labelwidth}%
			\addtolength{\leftmargin}{\labelsep}%
			}}%
{\end{list}}

\jyear{2021}%

\theoremstyle{thmstyleone}%
\newtheorem{theorem}{Theorem}
\newtheorem{proposition}[theorem]{Proposition}%

\theoremstyle{thmstyletwo}%
\newtheorem{example}{Example}%
\newtheorem{remark}{Remark}%
\newtheorem{lemma}{Lemma}
\theoremstyle{thmstylethree}%
\newtheorem{definition}{Definition}%
\begin{document}

\title{Conditional gradient method for vector optimization}

\author[1]{\fnm{Wang} \sur{Chen}}

\author[2]{\fnm{Xinmin} \sur{Yang}}

\author[3]{\fnm{Yong} \sur{Zhao}}

\affil[1]{\orgdiv{College of Mathematics}, \orgname{Sichuan University}, \orgaddress{\city{Sichuan}, \postcode{610065}, \country{China}}}

\affil[2]{\orgdiv{School of Mathematical Sciences}, \orgname{Chongqing Normal University}, \orgaddress{ \city{Chongqing}, \postcode{401331}, \country{China}}}

\affil[3]{\orgdiv{College of Mathematics and Statistics}, \orgname{Chongqing Jiaotong University}, \orgaddress{\ \city{Chongqing}, \postcode{610101}, \country{China}}}


\abstract{In this paper, we propose a conditional gradient method for solving constrained vector optimization problems with respect to a partial order induced by a closed, convex and pointed cone with nonempty interior. When the partial order under consideration is the one induced by the non-negative orthant, we regain the method for multiobjective optimization recently proposed by Assun\c{c}\~{a}o et al. (Comput Optim Appl 78(3):741--768, 2021). In our method, the construction of auxiliary subproblem is based on the well-known oriented distance function. Three different types of step size strategies (Armijio, adaptative and nonmonotone) are considered. Without any assumptions, we prove that stationarity of accumulation points of the sequences produced by the proposed method equipped with the Armijio or the nonmonotone step size rule. To obtain the convergence result of the method with the adaptative step size strategy, we introduce an useful cone convexity condition which allows to circumvent the intricate question of the Lipschitz continuity of Jocabian for the objective function. This condition helps us generalize the classical descent lemma to the vector optimization case. Under suitable convexity assumptions for the objective function, it is proved that all accumulation points of any generated sequences obtained by our method are weakly efficient solutions.}




\maketitle

\section{Introduction}\label{sec1}

In vector optimization, one considers a vector-valued function $F:\mathbb{R}^{n}\rightarrow\mathbb{R}^{m}$, and the partial order defined by a closed, convex and pointed cone $C\subseteq \mathbb{R}^{m}$ with nonempty interior, denoted by ${\rm int}(C)$. The partial order $\preceq_{C}$ $(\prec_{C})$ induced by $C$ in $\mathbb{R}^{m}$ is given by $y_{1}\preceq_{C} y_{2}$ ($y_{1}\prec_{C} y_{2}$) if and only if $y_{2}-y_{1}\in C$ ($y_{2}-y_{1}\in{\rm int}(C)$). We are interested in the following vector optimization problem:
\begin{equation}\label{mop}
	\begin{aligned}
		\text{min}_{C}&\quad F(x)=(f_{1}(x),f_{2}(x),...,f_{m}(x))^{\top}\\
		\text{s.t.}&\quad x\in \Omega,
	\end{aligned}
\end{equation}
where $F:\mathbb{R}^{n}\rightarrow\mathbb{R}^{m}$ is a vector-valued function and $\Omega\subseteq\mathbb{R}^{n}$ is a nonempty set. Note that ``$\min_{C}$'' represents the optimum with respect to the cone $C$ and the superscript ``${\top}$'' denotes the transpose. A point $x\in\Omega$ is called a \emph{weakly efficient} solution of problem (\ref{mop}) if there exists no $x^{*}\in\Omega$ such that $F(x^{*})\prec_{C} F(x)$ (see \cite{L1989,J2011}).

We will refer to \eqref{mop} as a unconstrained vector optimization problem if $\Omega=\mathbb{R}^{n}$. Otherwise, it is a constrained vector optimization problem. When $m=1$ and $\preceq_{C}$ is the usual linear order in $\mathbb{R}$, the problem \eqref{mop} is called the scalar optimization problem. When $C=\mathbb{R}_{+}^{m}$, where $\mathbb{R}_{+}^{m}$ is the non-negative orthant of $\mathbb{R}^{m}$, the problem \eqref{mop} is called the multicriteria or multiobjective optimization problem. Some practical applications of multiobjective optimization can be found in engineering \cite{R_m2013}, finance \cite{Z_m2015}, environment analysis \cite{F_o2001}, management science \cite{T_a2010}, machine learning \cite{J_m2006}, etc. In recent years, many iterative methods for solving scalar optimization problems have been extended to the multiobjective optimization case, such as steepest descent \cite{FS2000}, Newton \cite{FDS2009, WHW2019}, subgradient \cite{da_a2013}, quasi-Newton \cite{QGC2011}, trust region \cite{QGL2013,C_t2016} methods, and so on. It is worth mentioning that even though the vast majority of real-world problems formulated as vector-valued problems copes with the point-wise partial order $\mathbb{R}_{+}^{m}$. However, as reported in \cite{DS2005,DI2004,FD2013}, there are many others that require preference orders induced by closed convex cones rather than the non-negative orthant. Such cones have been recently analyzed in real-life problems, for instance, the issues concerning portfolio selection in security markets \cite{A_e2004,A_g2004}, the vector approximation problem and the cooperative $n$ player differential games \cite{J2011}. Therefore, it is necessary to focus our attention on vector optimization problems.

One of the main solution strategies for vector optimization problems is the so-called scalarization method whose core idea is to convert a target vector optimization problem into a scalar one (see \cite{L1989,G_o2006,J2011,A_v2018}), in such a way that the optimal solution of the new problem is also solution for the original one. Another strategy is the iterative methods, which directly extend the existing scalar optimization and multiobjective optimization methods to the vector optimization case. For example, the proximal point method used in scalar optimization have been successfully extended by \cite{V_a2011} to solve vector optimization problems on nonpolyhedral set. Chen \cite{C_g2011} proposed a class of generalized viscosity methods to solve vector optimization problem with nonsmooth objective function. Bello Cruz \cite{B_a2013} considered an extension of the projected subgradient method to solve convex constrained vector optimization. In \cite{L_n2018,GP2020}, the authors proposed nonlinear conjugate gradient methods for solving unconstrained vector optimization problems. The vector extensions of the Fletcher--Reeves, conjugate descent, Dai--Yuan, Polak--Ribi\`{e}re--Polyak, Hestenes--Stiefel and Hager--Zhang parameters were considered. Gra\~{n}a Drummond et al. \cite{DI2004,FD2013} provided an extension of the projected gradient method to solve constrained vector optimization problems. In addition, on the basis of Fliege and Svaiter's work \cite{FS2000}, Gra\~{n}a Drummond and Svaiter \cite{DS2005} regained the steepest descent method for solving unconstrained vector optimization problems. Compared with \cite{FS2000}, the remarkable feature in \cite{DS2005} is that the subproblem is given with the help of a gauge function, which has a more general form. Note that, when $C=\mathbb{R}_{+}^{m}$,  vector versions of projected gradient and steepest descent methods are consistent with the ones in \cite{FS2000} and \cite{FD2014}, respectively. Following Gra\~{n}a Drummond and Svaiter's \cite{DS2005} research path, vector versions of Newton method \cite{DRS2014,LC2014} have been proposed for solving vector optimization problems. It is worth noting that, as far as we know, Chen et al. \cite{C_a2009, C_a2009a}, Chuong and Yao \cite{CY2012}, Bonnel \cite{BIS2005}, Chuong \cite{C2010,C2013}, Bo\c{t} and Grad \cite{BG2018} also proposed some extension methods for solving vector optimization problems in infinite-dimensional settings. Therefore, the use of extensions of the iterative methods in scalar/multiobjective optimization to vector optimization is currently a promising research.

We should be noted that the most of aforementioned methods for multiobjective/vector optimization problems are equipped with the monotone line search rule. It is well-known that, in scalar optimization, the enforcing monotonicity of function values makes the corresponding methods to converge slower in the minimization process, especially when the iterate is in the bottom of a narrow curved valley \cite{G_a1986}. To overcome the shortcoming caused by the monotone search rule, numerous nonmonotone line methods were proposed and then verified numerically in the literature (see \cite{G_a1986,Z_a2004,G_i2008,A_o2017}, but not limited to them).  
In recent years, the studies on algorithms based on nonmonotone schemes for solving multiobjective optimization problems have attracted more attention.  Mita et al. \cite{M_n2019} extended the max-type \cite{G_a1986} and the average-type \cite{Z_a2004} nonmonotone line searches for scalar optimization to the multiobjective setting and then successfully applied them to the multiobjective versions of steepest descent and Newton methods. In \cite{M_a2020}, the authors considered a quasi-Newton method based on the nonmonotone line search technique of \cite{Z_a2004} for unconstrained multiobjective optimization problem with strongly convex objective function. Zhao et al. \cite{Z_l2021} proposed a projected gradient method equipped with the nonmonotone line search procedure given by Hager and Zhang \cite{Z_a2004} for constrained multiobjective optimization problems. Ramirez and Sottosanto \cite{R_n2022}  presented a trust region algorithm with a nonmonotone technique for solving the unconstrained multiobjective optimization problems. Numerical experiments in \cite{M_a2020,Z_a2004,R_n2022} suggests that the nonmonotone line searches tend to be more efficient than monotone ones. For vector optimization problem, Qu et al. \cite{Q_n2017} utilized the max-type \cite{G_a1986} nonmonotone strategy to give two nonmonotone gradient algorithms and then established both the global and local convergence results for the two algorithms. So far, to the best of our knowledge, there are few researches about extending nonmonotone strategy to vector optimization where the partial order is a general cone. Therefore, it is necessary for us to study other types of nonmonotone techniques in vector optimization.

Very recently, Assun\c{c}\~{a}o et al. \cite{AFP2021} introduced a multiobjective version of the classical conditional gradient (also known as Frank-Wolfe) method presented in \cite[pp.378]{B2017} for constrained multiobjective optimization problems. Convergence analysis and iteration-complexity bounds of the method were detailedly studied under suitable assumptions. Their numerical experiments illustrate the effectiveness of the method. It is noteworthy that, in the conclusion part of \cite{AFP2021}, Assun\c{c}\~{a}o et al. put forward an interesting question:

\emph{``A natural question is whether it is possible to analyze the conditional gradient method for vector optimization problems, i.e., when the partial order in $\mathbb{R}^{m}$ is induced by other underlying cones instead of the non-negative orthant.''}

The purpose of this paper is to take a step further on the direction of \cite{AFP2021} and give a pertinent solution to the above question. Based on the research line of literature \cite{DS2005,DRS2014,LC2014,DI2004,FD2013,CY2012,C2013}, to solve problem \eqref{mop}, we directly and successfully extend the conditional gradient method proposed in \cite{AFP2021}. At each iteration of our method, the descent direction is the difference between the current iterate and a optimal solution of a auxiliary subproblem. We should mention that the subproblem in our method is based on the well-known oriented distance function \cite{JBH1979,AZ2003}. When the partial order cone is the nonnegative orthant and the norm is the infinite norm, our subproblem's form turns out to be the very same proposed by \cite{AFP2021} for the multiobjective case. Our method is considered with three different step size strategies. Firstly, a Armijo line search rule is conducted to choose the step size at every iteration and ensures that the sequence of objective function values decreases in the partial order induced by cone $C$. The convergence property that requires no any conditions on the objective function is established. Secondly, we analyze the proposed method with a adaptative step size rule. It is to emphasize that, in convergence analysis, we circumvent the intricate question of Lipschitz continuity of Jacobian for objective function by using an elegant and easy to check $C$-convexity condition. The useful condition is inspired by the work of Bauschke et al. \cite{BBT2017}, and as far as we know, there is no research about extending the condition to vector optimization. Finally, based on the nonmonotone line search given by Gu and Mo \cite{G_i2008} in the scalar context, we introduce a new nonmonotone technique in vector optimization, and then prove that the line search also guarantees the convergence of our method. For convex objective functions with respect to the partial order $C$, we show all accumulation points of any generated sequences are weakly efficient solutions.

The outline of this paper is as follows. Section 2 presents some basic definitions, notations and auxiliary results which will be used in the sequel. In Section 3, we give the general scheme for conditional gradient method for constrained vector optimization problems. In Section 4, we analyze the convergence of the proposed method with three different step size strategies. Finally, in Section 5, some conclusions and remarks are presented.

\section{Preliminaries}
For a nonempty set $X\subseteq \mathbb{R}^{m}$, the boundary of $X$ is denoted by ${\rm bd}(X)$. Let $I=\{1,2,\ldots,m\}$. For a compact set $X$, the diameter of $X$ is given by ${\rm diam}(X)=\mathop{\max}_{x,y\in X}\|x-y\|$, where $\|\cdot\|$ denotes the norm in $\mathbb{R}^{m}$. 

We now recall the concept of oriented distance function (also called assigned distance function or Hiriart-Urruty function), which was proposed by Hiriart-Urruty \cite{JBH1979} to investigate optimality conditions of nonsmooth optimization problems from the geometric point of view. The oriented distance function has been extensively used in several works, such as scalarization for vector optimization \cite{AZ2003, MMR2005,L_m2009}, robustness for multiobjective optimization \cite{A_c2019}, optimality conditions for vector optimization  \cite{GHY2012}, optimality conditions for set-valued optimization \cite{ZCY2019}, and so on. Herein, we consider the  oriented distance function in $\mathbb{R}^{m}$.

\begin{definition}\label{dis_def}
	\cite{JBH1979,AZ2003}
	Let $A$ be a subset of $\mathbb{R}^{m}$. The function $\Delta_{A}:\mathbb{R}^{m}\rightarrow\mathbb{R}\cup\{\pm\infty\}$, defined by
	\begin{equation}\label{dis_fun}
		\Delta_{A}(y)=d_{A}(y)-d_{\mathbb{R}^{m}\backslash A}(y),\quad\forall y\in \mathbb{R}^{m},
	\end{equation}
	is called the oriented distance function, where $d_{A}(y):=\inf\{\|y-a\|:a\in A\}$ stands for the distance function from $y\in \mathbb{R}^{m}$ to the set $A$.
\end{definition}

We give the following example to illustrate the function $\Delta_{A}$.

\begin{example}\label{ex1}
	\begin{enumerate}[(i)]
		\item If we consider the norm $\|y\|_{2}=(\sum_{i=1}^{m}y_{i}^{2})^{1/2}$ in $\mathbb{R}^{m}$ and $A=\{y\in\mathbb{R}^{m}:\|y\|_{2}\leq1\}$, then $\triangle_{A}(y)=\|y\|_{2}-1$.
		\item If we take the norm $\|y\|_{\infty}=\max_{i\in I}\lvert y_{i}\rvert$ and $A=-\mathbb{R}_{+}^{m}$, then
		$\triangle_{A}(y)=\max_{i\in I}y_{i}$.
		\item Consider the norm $\|y\|_{2}$ in $\mathbb{R}^{2}$ and the partial order $C=\{y=(y_{1},y_{2})\in\mathbb{R}^{2}:0\leq y_{1}+y_{2},0\leq y_{2}\}$.  Let $A=-C$ and
		\begin{equation*}
			\begin{aligned}
				B_{1}&=\{y=(y_{1},y_{2})\in\mathbb{R}^{2}:y_{1}\leq 0,y_{2}> 0\},\\
				B_{2}&=\{y=(y_{1},y_{2})\in\mathbb{R}^{2}:y_{1}-y_{2}\leq 0,y_{1}> 0\},\\
				B_{3}&=\{y=(y_{1},y_{2})\in\mathbb{R}^{2}:y_{1}-y_{2}> 0,y_{1}+y_{2}>0\},\\
				B_{4}&=\{y=(y_{1},y_{2})\in\mathbb{R}^{2}:y_{1}-y_{2}>0,y_{1}+y_{2}\leq0\},\\
				B_{5}&=\{y=(y_{1},y_{2})\in\mathbb{R}^{2}:y_{1}-y_{2}\leq0,y_{2}\leq 0\}.\\
			\end{aligned}
		\end{equation*}
		Clearly, $A=B_{4}\cup B_{5}$ and $Y\backslash A=B_{1}\cup B_{2}\cup B_{3}$. By a direct calculation, we have
		\begin{equation*}
			d_{A}(y)=\left\{
			\begin{array}{lll}
				y_{2}, &\quad{\rm if}\quad y\in B_{1},\\
				\|y\|_{2}, &\quad{\rm if}\quad y\in B_{2},\\
				\lvert y_{1}+y_{2}\rvert/\sqrt{2},&\quad{\rm if}\quad y\in B_{3}
			\end{array}
			\right.
		\end{equation*}
		and
		\begin{equation*}
			d_{\mathbb{R}^{2}\backslash A}(y)=
			\left\{\begin{array}{lll}
				\lvert y_{1}+y_{2}\rvert/\sqrt{2}, &\quad{\rm if}\quad y\in B_{4},\\
				\lvert y_{2}\rvert, &\quad{\rm if}\quad y\in B_{5}.
			\end{array}\right.
		\end{equation*}
	\end{enumerate}
\end{example}

As pointed out by Ansari et al. \cite[Remark 2.1]{A_c2019}, the oriented distance function can easily be implemented by using the fast marching method, fast sweeping method and level set methods (see \cite{Z_a2005}). In this paper, for our purposes, let $A=-C$ in Definition \ref{dis_def}. For the sake of convenience, in what follows, we will use 
\begin{equation}\label{varphi}
	\varphi_{_{C}}(y)=\triangle_{-C}(y),\quad \forall y\in \mathbb{R}^{m}.
\end{equation}
According to \cite[Proposition 3.2]{AZ2003} and the fact that $C$ is a closed, convex and pointed cone with nonempty interior, we immediately have the following properties related to $\varphi_{_{C}}$, which will be used in our subsequent analysis.

\begin{lemma}{\rm\cite{AZ2003}}\label{varphi_property}
	Let $\varphi_{_{C}}(\cdot)$ be defined in (\ref{varphi}). Then, the following statements hold:
	\begin{enumerate}
		\item $\varphi_{_{C}}$ is a $1$-Lipschitzian function;
		\item $\varphi_{_{C}}(y)<0$ for any $y\in-{\rm int}(C)$, $\varphi_{_{C}}(y)=0$ for any $y\in{\rm bd}(-C)$, and $\varphi_{_{C}}(y)>0$ for any $y\in{\rm int}(\mathbb{R}^{m}\backslash(-C))$; \label{varphi_property_ii}
		\item $\varphi_{_{C}}$ is convex;
		\item $\varphi_{_{C}}$ is positively homogeneous;
		\item For all $y_{1},y_{2}\in\mathbb{R}^{m}$, $\varphi_{_{C}}(y_{1}+y_{2})\leq\varphi_{_{C}}(y_{1})+\varphi_{_{C}}(y_{2})$ and
		$\varphi_{_{C}}(y_{1})-\varphi_{_{C}}(y_{2})\leq\varphi_{_{C}}(y_{1}-y_{2})$;
		\item Let $y_{1},y_{2}\in\mathbb{R}^{m}$. If $y_{1}\preceq_{C} (\prec_{C}) y_{2}$, then $\varphi_{_{C}}(y_{1})\leq (<) \varphi_{_{C}}(y_{2})$.
	\end{enumerate}
\end{lemma}

From now on, we assume that $F$ is a continuously differentiable function and the constraint set $\Omega\subseteq\mathbb{R}^{n}$ is nonempty, compact and convex. Given $x=(x_{1},x_{2},\ldots,x_{n})\in\mathbb{R}^{n}$, the Jacobian of $F$ at $x$, denoted by $JF(x)$, is a matrix of order $m\times n$ whose entries are defined by
$$(JF(x))_{i,j}=\frac{\partial f_{i}}{\partial x_{j}}(x),$$ 
where $i\in I$ and $j\in\{1,2,\ldots,n\}$. We may represent it by
$$JF(x)=[\nabla f_{1}(x)\:\nabla f_{2}(x)\:\ldots\:\nabla f_{m}(x)]^{\top},\quad x\in\mathbb{R}^{n}.$$

A necessary, but not sufficient, first-order optimality condition for problem (\ref{mop}) at $x\in\Omega$, is
\begin{equation}\label{fir_ord_opt}
	JF(x)(\Omega-x)\cap(-{\rm int}(C))=\emptyset,
\end{equation}
where $JF(x)(\Omega-x)=\{JF(x)(s-x):s\in\Omega\}$ and
$$JF(x)(s-x)=(\langle\nabla f_{1}(x),s-x\rangle,\langle\nabla f_{2}(x),s-x\rangle,\ldots,\langle\nabla f_{m}(x),s-x\rangle)^{\top}.$$
Obviously, (\ref{fir_ord_opt}) is equivalent to
$JF(x)(s-x)\notin-{\rm int}(C)$ for any $s\in\Omega$.

\begin{definition}\normalfont\label{pare_sta}
	A point $\hat{x}\in\Omega$ satisfying (\ref{fir_ord_opt}) is called a stationary point of problem (\ref{mop}).
\end{definition}

\begin{remark}\normalfont\label{non_stationary_point}
	From Lemma \ref{varphi_property}(ii), we can obtain a equivalent characterization of a stationary point $\hat{x}$ of problem (\ref{mop}), i.e., $$\varphi_{_{C}}(JF(\hat{x})(s-\hat{x}))\geq0,\quad\forall s\in\Omega.$$
\end{remark}

\begin{remark}\normalfont
	\begin{enumerate}
		\item If $m=1$ and $C=\mathbb{R}^{1}_{+}$, then we retrieve the classical stationary condition for constrained scalar optimization problem, i.e., $\langle\nabla f_{1}(\hat{x}),s-\hat{x}\rangle\geq0$ for all $s\in\Omega$.
		\item If we consider the norm $\|\cdot\|_{\infty}$ and $C=\mathbb{R}_{+}^{m}$, then by Remark \ref{ex1} and Example \ref{ex1}(ii), Definition \ref{pare_sta} is the same as the notion presented in \cite[pp.744]{AFP2021}.
	\end{enumerate}
\end{remark}

\begin{remark}\normalfont\label{rem3.3}
	Note that if $\hat{x}\in\Omega$ is not a stationary point of problem (\ref{mop}), then there exists $\hat{s}\in\Omega$ such that $JF(\hat{x})(\hat{s}-\hat{x})\in-{\rm int}(C)$, i.e., $\varphi_{_{C}}(JF(\hat{x})(\hat{s}-\hat{x}))<0$ from Lemma \ref{varphi_property}(ii). In this case, as analyzed in \cite[pp.665]{DRS2014}, we can assert that $\hat{s}-\hat{x}$ is a descent direction for $F$.
\end{remark}

We now recall the concept of $C$-convexity of $F$, which is presented in \cite[Definition 2.4]{J2011}. The objective function $F$ is called $C$-convex on $\Omega$ if 
$$F(\lambda x+(1-\lambda)y)\preceq_{C} \lambda F(x)+(1-\lambda)F(y),$$
for all $x,y\in\Omega$ and all $\lambda\in[0,1]$.  Since $F$ is continuously differentiable, a equivalent characterization of $C$-convexity of $F$ is
$$JF(y)(x-y)\preceq_{C} F(x)-F(y)$$
for all $x,y\in\Omega$ (see \cite[Theorem 2.20]{J2011}).

We conclude this section by giving the relationship between stationary point and weakly efficient solution. The proof of this property can be similarly analyzed from \cite[pp.410]{DS2005} and we omit the process here.

\begin{theorem}\label{stationary_weakly_pareto}
	\begin{enumerate}
		\item If $\hat{x}\in\Omega$ is a weakly efficient solution of problem (\ref{mop}), then $\hat{x}\in\Omega$ is a stationary point.
		\item If $F$ is $C$-convex on $\Omega$ and $\hat{x}\in\Omega$ is a stationary point of problem (\ref{mop}), then $\hat{x}$ is a weakly efficient solution.
	\end{enumerate}
\end{theorem}

\section{Conditional gradient method for vector optimization}
In this section, we propose a general scheme of conditional gradient (for short, CondG) method for vector optimization problem \eqref{mop}. 

For a given $x\in\Omega$, we introduce a useful auxiliary function $\psi_{x}:\Omega\rightarrow\mathbb{R}$ defined by
\begin{equation}\label{psix}
	\psi_{x}(s)=\varphi_{_{C}}(JF(x)(s-x)),\quad s\in\Omega.
\end{equation}

For $x\in\Omega$, in order to obtain the descent direction for $F$ at $x$, we need to consider the following auxiliary scalar optimization problem:
\begin{equation}\label{sca_pro1}
	\min_{s\in\Omega}\;\psi_{x}(s).
\end{equation}
It follows from Lemma \ref{varphi_property}(iii) that $\psi_{x}$ defined in (\ref{psix}) is a convex function. This, combined with the fact that $\Omega$ is a nonempty, compact and convex set, gives that problem (\ref{sca_pro1}) admits an optimal solution (possibly not unique) on $\Omega$. We denote the optimal solution of problem (\ref{sca_pro1}) by $s(x)$, i.e.,
\begin{equation}\label{opt_sol}
	s(x)\in\mathop{\rm argmin}_{s\in\Omega}\;\psi_{x}(s).
\end{equation}
and the optimal value of problem (\ref{sca_pro1}) is denoted by $v(x)$, i.e.,
\begin{equation}\label{opt_val}
	v(x)=\psi_{x}(s(x)).
\end{equation}

For multiobjective optimization, where $C=\mathbb{R}_{+}^{m}$, if we consider the norm $\|\cdot\|_{\infty}$, then $s(x)$ can be computed by solving
\begin{equation}\label{sca_pro11}
	\min_{s\in\Omega}\max_{i\in I}\;\langle\nabla f_{i}(x),s-x\rangle,
\end{equation}
which was also considered in \cite{AFP2021}.

According to Remark \ref{rem3.3}, we formally give the search direction for the objective function $F$ at $x$.

\begin{definition}\normalfont
	For any given point $x\in\Omega$, the search direction of the conditional gradient method for $F$ at $x$ is defined as
	\begin{equation}\label{search_direction}
		d(x)=s(x)-x,
	\end{equation}
	where $s(x)$ is given by (\ref{opt_sol}).
\end{definition}

The following property gives a characterization of stationarity in terms of  $v(\cdot)$, which is crucial for convergence analysis and the stopping criteria of our algorithm.

\begin{proposition}\label{pa_sta_equ}
	Let $v:\Omega\rightarrow\mathbb{R}$ be defined in (\ref{opt_val}). Then, the following statements hold:
	\begin{enumerate}
		\item $v(x)\leq0$ for all $x\in\Omega$;
		\item  $x\in\Omega$ is a stationary point of problem (\ref{mop}) if and only if $v(x)=0$.
	\end{enumerate}
\end{proposition}
\begin{proof}
	(i) Since $x\in\Omega$, by (\ref{opt_sol}) and (\ref{opt_val}) and observing that $\varphi_{_{C}}(0)=0$, we have $$v(x)=\min_{s\in\Omega}\psi_{x}(s)\leq\psi_{x}(x)=\varphi_{_{C}}(JF(x)(x-x))=\varphi_{_{C}}(0).$$
	
	(ii) Necessity. Suppose that $x\in\Omega$ is a stationary point of problem (\ref{mop}). Then, it follows from Remark \ref{non_stationary_point} that $\varphi_{_{C}}(JF(x)(s-x))\geq0$ for any $s\in\Omega$. By (\ref{opt_sol}), we have $s(x)\in\Omega$. Hence, $v(x)=\varphi_{_{C}}(JF(x)(s(x)-x))\geq0$. This, combined with (i), yields that $v(x)=0$.
	
	Sufficiency. Let $v(x)=0$. According to (\ref{opt_val}), we obtain
	$$0=v(x) \leq\psi_{x}(s)=\varphi_{_{C}}(JF(x)(s-x))$$
	for all $s\in\Omega$, which implies that $x$ is a stationary point of problem (\ref{mop}).
\end{proof}

\begin{remark}\normalfont\label{pa_sta_equ1}
	It is obvious from Proposition \ref{pa_sta_equ} that $x$ is not a stationary point of problem (\ref{mop}) if and only if $v(x)<0$.
\end{remark}

\begin{proposition}\label{opt_val_continuous}
	Let $v:\Omega\rightarrow\mathbb{R}$ be defined in (\ref{opt_val}). Then, $v$ is continuous on $\Omega$.
\end{proposition}
\begin{proof}
	Take $x\in\Omega$ and let $\{x^{k}\}$ be a sequence in $\Omega$ such that $\lim_{k\rightarrow\infty} x^{k}=x$. In order to obtain the continuity of $v$ on $\Omega$, it is sufficient to prove that $\lim_{k\rightarrow\infty}v(x^{k})=v(x)$, i.e., 
	\begin{equation}\label{continuous}
		\mathop{\lim\sup}_{k\rightarrow\infty} v(x^{k})\leq v(x)\leq\mathop{\lim\inf}_{k\rightarrow\infty} v(x^{k}).
	\end{equation}
	
	Since $s(x)\in\Omega$, using (\ref{opt_sol}) and (\ref{opt_val}), we can obtain for all $k$,
	\begin{equation}\label{continous_theorem1}
		v(x^{k})=\varphi_{_{C}}(JF(x^{k})(s(x^{k})-x^{k}))\leq\varphi_{_{C}}(JF(x^{k})(s(x)-x^{k})).
	\end{equation}
	Since $F$ is continuously differentiable and $\varphi_{_{C}}$ is continuous as presented in Lemma \ref{varphi_property}(i), taking $\lim\sup_{k\rightarrow\infty}$ on both sides of inequality in (\ref{continous_theorem1}), we have
	\begin{equation}\label{continous_theorem2}
		\mathop{\lim\sup}_{k\rightarrow\infty} v(x^{k})\leq\varphi_{_{C}}(JF(x)(s(x)-x))=v(x).
	\end{equation}
	
	Let us show that $v(x)\leq\mathop{\lim\inf}_{k\rightarrow\infty} v(x^{k})$. Obviously, we have
	\begin{equation}\label{continous_theorem3}
		\begin{aligned}
			v(x)&=\min\{\psi_{x}(s),s\in\Omega\}\\
			&\leq\psi_{x}(s(x^{k}))\\
			&=\varphi_{_{C}}(JF(x)(s(x^{k})-x))\\
			&=\varphi_{_{C}}(JF(x)(s(x^{k})-x^{k}+x^{k}-x))\\
			&=\varphi_{_{C}}(JF(x)(s(x^{k})-x^{k})+JF(x)(x^{k}-x))\\
			&\leq\varphi_{_{C}}(JF(x)(s(x^{k})-x^{k}))+\varphi_{_{C}}(JF(x)(x^{k}-x)),
		\end{aligned}
	\end{equation}
	where the last inequality follows from Lemma \ref{varphi_property}(v). Taking $\lim\inf_{k\rightarrow\infty}$ in (\ref{continous_theorem3}), we get
	\begin{equation}\label{ifff}
		\begin{aligned}
			v(x)&\leq\mathop{\lim\inf}_{k\rightarrow\infty}\varphi_{_{C}}(JF(x)(s(x^{k})-x^{k}))\\
			&=\mathop{\lim\inf}_{k\rightarrow\infty}(v(x^{k})+\varphi_{_{C}}(JF(x)(s(x^{k})-x^{k}))-\varphi_{_{C}}(JF(x^{k})(s(x^{k})-x^{k})))\\
			&\leq\mathop{\lim\inf}_{k\rightarrow\infty}(v(x^{k})+\|JF(x)(s(x^{k})-x^{k})-JF(x^{k})(s(x^{k})-x^{k})\|)\\
			&=\mathop{\lim\inf}_{k\rightarrow\infty}(v(x^{k})+\|(JF(x)-JF(x^{k}))(s(x^{k})-x^{k})\|)\\
			&\leq\mathop{\lim\inf}_{k\rightarrow\infty}(v(x^{k})+\|JF(x)-JF(x^{k})\|\|s(x^{k})-x^{k}\|),
		\end{aligned}
	\end{equation}
	where the second inequality follows from Lemma \ref{varphi_property}(i). Since $s(x^{k}),x^{k}\in\Omega$, it follows that $\|s(x^{k})-x^{k}\|\leq{\rm diam}(\Omega)<\infty$. This, combined with the continuously differentiability of $F$ and (\ref{ifff}), we get $v(x)\leq\mathop{\lim\inf}_{k\rightarrow\infty} v(x^{k})$.  Altogether, (\ref{continuous}) holds. Consequently, $v$ is continuous on $\Omega$.
\end{proof}

Based on the previous discussions, we are now ready to describe the general scheme of the conditional gradient algorithm for solving problem (\ref{mop}). 
\vskip0.1in
\begin{newlist}{Step1: }
	\item[CondG algorithm.]
	\item[Step 0] Choose $x^{0}\in\Omega$. Compute $s^{0}$ and $v^{0}$ and initialize $k\leftarrow0$.
	\item[Step 1] If $v^{k}=0$, then \textbf{stop}.
	\item[Step 2] Compute $d^{k}=s^{k}-x^{k}$.
	\item[Step 3] Compute a step size $t_{k}\in(0,1]$ by a step size strategy and set $x^{k+1}=x^{k}+t_{k}d^{k}$.
	\item[Step 4] Compute $s^{k+1}$ and $v^{k+1}$, set $k\leftarrow k+1$, and go to \textbf{Step 1}.
\end{newlist}
\vskip0.1in

At $k$-th iteration, we solve problem (\ref{sca_pro1}) with $x=x^{k}$. Let us call $s^{k}=s(x^{k})$ and $v^{k}=v(x^{k})$ the optimal solution and optimal value of problem (\ref{sca_pro1}) at $k$-th iteration, respectively. The descent direction at $k$-th iteration is computed by $d^{k}=d(x^{k})=s^{k}-x^k$. If $v^{k}\neq 0$, then we can use $d^{k}$ with a step size strategy to look for a new solution $x^{k+1}$ which dominates $x^{k}$. 

Observe that the CondG algorithm can terminate with a stationary point in a finite number of iterations or, it produces an infinite sequence of nonstationary points. In convergence analysis, we will assume that $\{x^{k}\}$, $\{s^{k}\}$, $\{d^{k}\}$ and $\{t^{k}\}$ are infinite sequences. Therefore, using Remark \ref{pa_sta_equ1}, we have $v^{k}<0$ for all $k$. A simple and elementary fact that the proposed algorithm generates feasible sequences is given below.

\begin{remark}\normalfont\label{feasible}
	It is noteworthy that the sequence $\{x^{k}\}$ generated by the CondG algorithm contains in $\Omega$. Indeed, since $x^{0}\in\Omega$, $s^{k}\in\Omega$ and $t_{k}\in(0,1]$, it follows from inductive arguments that $x^{k}\in\Omega$ for all $k$.
\end{remark}

\section{Convergence analysis}

It is worth noting that the choice for computing the step size $\alpha_{k}$ at Step 3 is crucial. In this section, we establish the main convergence properties of the CondG algorithm with three different step size strategies.

\subsection{Analysis of CondG algorithm with Armijo step size}

\noindent\textbf{Armijo step size} Let $\beta\in(0,1)$, $\delta\in(0,1)$ and $\tau\in(0,1]$. Choose $t_{k}$ as the largest value in $\{\tau,\delta\tau,\delta^{2}\tau,\ldots\}$ such that
\begin{equation}\label{armijio}
	F(x^{k}+t_{k} d^{k})\preceq_{C} F(x^{k})+\beta t_{k} JF(x^{k})d^{k}.
\end{equation}

Other vector optimization methods, including steepest descent, Newton and projected gradient methods, that use the Armijo rule can be found in \cite{DS2005,LC2014,DI2004,FD2013,CY2012,C2013}. The next proposition allows us to implement the line search \eqref{armijio} at Step 3 of the CondG algorithm.

\begin{proposition}\label{des_lemma0}
	Let $\beta\in(0,1)$ and $s(x)$ be defined in (\ref{opt_sol}) and $JF(x)(s(x)-x)\prec_{C}0$. Then there exists some $\hat{t}\in(0,1]$ such that
	\begin{equation*}
		F(x+t(s(x)-x))\prec_{C} F(x)+t\beta JF(x)(s(x)-x).
	\end{equation*}
	for any $t\in(0,\hat{t}]$.
\end{proposition}
\begin{proof}
	It is analogous to the proof of \cite[Proposition 2.1]{DS2005}.
\end{proof}

We present some properties related to the points which are iterated by the CondG algorithm with the Armijo step size condition \eqref{armijio}.

\begin{proposition}\label{properties}
	For all $k$, we have
	\begin{enumerate}
		\item $F(x^{k+1})\prec_{C} F(x^{k})$;
		\item $\sum_{i=0}^{k}t_{i}\lvert v^{i}\rvert\leq\beta^{-1}(\varphi_{_{C}}(F(x^{0}))-\varphi_{_{C}}(F(x^{k+1})))$.
	\end{enumerate}
\end{proposition}
\begin{proof}
	(i) The result follows from \eqref{armijio} and the nonstationarity of $x^{k}$.
	
	(ii) For any $i$, by \eqref{armijio} and Lemma \ref{varphi_property}(vi)--(vi), we have
	\begin{equation}\label{sum}
		\begin{aligned}
			\varphi_{_{C}}(F(x^{i+1}))&\leq\varphi_{_{C}}(F(x^{i})+t_{i}\beta JF(x^{i})(s^{i}-x^{i}))\\
			&\leq\varphi_{_{C}}(F(x^{i}))+\varphi_{_{C}}(t_{i}\beta JF(x^{i})(s^{i}-x^{i}))\\
			&=\varphi_{_{C}}(F(x^{i}))+t_{i}\beta\varphi_{_{C}}(JF(x^{i})(s^{i}-x^{i}))\\
			&=\varphi_{_{C}}(F(x^{i}))+t_{i}\beta v^{i}.
		\end{aligned}
	\end{equation}
	From \eqref{sum} and observing that $v^{i}<0$, we obtain 
	\begin{equation*}
		t_{i}\lvert v^{i}\rvert=-t_{i} v^{i}\leq\frac{\varphi_{_{C}}(F(x^{i}))- \varphi_{_{C}}(F(x^{i+1}))}{\beta}.
	\end{equation*}
	Summing the above inequality from $i=0$ to $i=k$, the result is immediately derived.
\end{proof}

\begin{theorem}\label{convergence}
	Let $\{x^{k}\}$ be a sequence produced by the CondG algorithm with the Armijio step size condition \eqref{armijio}. Then, every accumulation point of $\{x^{k}\}$ is a stationary point of problem (\ref{mop}).
\end{theorem}
\begin{proof}
	Let $\hat{x}\in\Omega$ be a accumulation point of the sequence $\{x^{k}\}$. Then, there exists a subsequence $\{x^{k_{j}}\}$ of $\{x^{k}\}$ such that
	\begin{equation}\label{convergence0}
		\lim_{j\rightarrow\infty}x^{k_{j}}=\hat{x}.
	\end{equation}
	From Proposition \ref{opt_val_continuous} and (\ref{convergence0}), we have $v(x^{k_{j}})\rightarrow v(\hat{x})$ whenever $j\rightarrow\infty$. Here, it is sufficient to show that $v(\hat{x})=0$ in view of Proposition \ref{pa_sta_equ}(ii).
	
	Let $k=k_{j}$ in Proposition \ref{properties}(ii). Then
	\begin{equation*}\label{convergence1}
		\sum_{i=0}^{k_{j}}t_{i}\lvert v(x^{i})\rvert\leq\frac{\varphi_{_{C}}(F(x^{0}))-\varphi_{_{C}}(F(x^{k_{j}+1}))}{\beta}.
	\end{equation*}
	Taking $\lim_{j\rightarrow\infty}$ on both sides of the above inequality, we get $\sum_{i=0}^{\infty}t_{i}\lvert v(x^{i})\rvert<\infty$, which implies that $\lim_{k\rightarrow\infty}t_{k}v(x^{k})=0$, and in particular,
	\begin{equation}\label{convergence3}
		\lim_{j\rightarrow\infty} t_{k_{j}}v(x^{k_{j}})=0.
	\end{equation}
	Since $t_{k}\in(0,1]$ for all $k$, we have the following two alternatives:
	\begin{equation}\label{two_alternatives}
		{\rm (a)}\;\mathop{\lim\sup}_{j\rightarrow\infty} t_{k_{j}}>0\qquad{\rm or}\qquad{\rm (b)}\;\mathop{\lim\sup}_{j\rightarrow\infty} t_{k_{j}}=0.
	\end{equation}
	
	We first suppose that (\ref{two_alternatives})(a) holds. Then, there exists a subsequence $\{t_{k_{j_{i}}}\}$ of $\{t_{k_{j}}\}$ converging to some $\hat{t}>0$. And from (\ref{convergence0}), we have $\lim_{i\rightarrow\infty}x^{k_{j_{i}}}=\hat{x}$. Thus, (\ref{convergence3}) implies that $\lim_{i\rightarrow\infty}t_{k_{j_{i}}}v(x^{k_{j_{i}}})=0$, and furthermore, $\lim_{i\rightarrow\infty}v(x^{k_{j_{i}}})=0$. This, combined with Proposition \ref{opt_val_continuous}, gives that $0=\lim_{i\rightarrow\infty}v(x^{k_{j_{i}}})=v(\hat{x})$.
	
	We now consider (\ref{two_alternatives})(b). Clearly, $x^{k_{j}},s(x^{k_{j}})\in\Omega$ and $$\|d(x^{k_{j}})\|=\|s(x^{k_{j}})-x^{k_{j}}\|\leq{\rm diam}(\Omega)<\infty,$$
	i.e., the sequence $\{d(x^{k_{j}})\}$ is bounded. Now, we take subseqences $\{x^{k_{j_{i}}}\}$, $\{d(x^{k_{j_{i}}})\}$ and $\{t_{k_{j_{i}}}\}$ converging to $\hat{x}$, $d(\hat{x})$ and $0$, respectively. By (\ref{opt_val}) and Proposition \ref{properties}(i), we get
	\begin{equation*}\label{convergence4}
		\varphi_{_{C}}(JF(x^{k_{j_{i}}})d(x^{k_{j_{i}}}))=v(x^{k_{j_{i}}})<0.
	\end{equation*}
	Taking $\lim_{i\rightarrow\infty}$ on both sides of the above inequality and togethering with Proposition \ref{opt_val_continuous}, we have
	\begin{equation}\label{convergence4_0}
		v(\hat{x})\leq0.
	\end{equation}
	Take some fixed but arbitrary $l\in\mathbb{N}$, where $\mathbb{N}$ denotes the set of natural numbers. From $\lim_{i\rightarrow\infty}t_{k_{j_{i}}}=0$, we have $t_{k_{j_{i}}}<\tau^{l}$ for $i$ large enough. This shows that the Armijio condition is not satisfied at $x^{k_{j_{i}}}$ for $t=\tau^{l}$, that is,
	$$F(x^{k_{j_{i}}}+\tau^{l}d(x^{k_{j_{i}}}))\npreceq_{C} F(x^{k_{j_{i}}})+\tau^{l}\beta JF(x^{k_{j_{i}}})d(x^{k_{j_{i}}}),$$
	or, equivalently,
	$$F(x^{k_{j_{i}}}+\tau^{l}d(x^{k_{j_{i}}}))-F(x^{k_{j_{i}}})-\tau^{l}\beta JF(x^{k_{j_{i}}})d(x^{k_{j_{i}}})\notin-C,$$
	which means that
	\begin{equation}\label{convergence5}
		F(x^{k_{j_{i}}}+\tau^{l}d(x^{k_{j_{i}}}))-F(x^{k_{j_{i}}})-\tau^{l}\beta JF(x^{k_{j_{i}}})d(x^{k_{j_{i}}})\in\mathbb{R}^{m}\backslash(-C)={\rm int}(\mathbb{R}^{m}\backslash(-C)),
	\end{equation}
	where the equality holds in view of the closedness of $C$. By (\ref{convergence5}) and Lemma \ref{varphi_property}(ii), we have
	\begin{equation}\label{convergence6}
		\varphi_{_{C}}(F(x^{k_{j_{i}}}+\tau^{l}d(x^{k_{j_{i}}}))-F(x^{k_{j_{i}}})-\tau^{l}\beta JF(x^{k_{j_{i}}})d(x^{k_{j_{i}}}))>0.
	\end{equation}
	Since $F$ is continuously differentiable and $\varphi_{_{C}}$ is continuous, taking $\lim_{i\rightarrow\infty}$ in (\ref{convergence6}), we obtain
	\begin{equation}\label{convergence7}
		\varphi_{_{C}}(F(\hat{x}+\tau^{l}d(\hat{x}))-F(\hat{x})-\tau^{l}\beta JF(\hat{x})d(\hat{x}))\geq0,\quad\forall l\in\mathbb{N}.
	\end{equation}
	According to (\ref{convergence7}), Proposition \ref{des_lemma0} and Lemma \ref{varphi_property}(ii), we can obtain $JF(\hat{x})d(\hat{x})\nprec_{C}0$, i.e., $JF(\hat{x})d(\hat{x})\in\mathbb{R}^{m}\backslash(-{\rm int}(C))$. Thus, we have $v(\hat{x})=\varphi_{_{C}}(JF(\hat{x})d(\hat{x}))\geq0$ from Lemma \ref{varphi_property}(ii). This, combined with (\ref{convergence4_0}), yields that $v(\hat{x})=0$.
\end{proof}

It follows from Theorems \ref{stationary_weakly_pareto} and \ref{convergence} that the following result holds.
\begin{theorem}
	If $F$ is $C$-convex on $\Omega$, then the sequence $\{x^{k}\}$ produced by the CondG algorithm with the Armijio step size condition \eqref{armijio} converges to a weakly efficient solution of problem (\ref{mop}).
\end{theorem}

\subsection{Analysis of CondG algorithm with adaptative step size}

In this section, we analyze the convergence of the CondG algorithm with an adaptative step size rule. The differences from the work \cite{AFP2021} are the fact that we not only consider the general partial order but also do not require the Jacobian $JF$ of the objective function to be Lipschitz continuous.


In this sequel, we let $e=(e_{1},e_{2},\ldots,e_{m})^{\top}\in{\rm int}(C).$ Let us recall the useful concept that the Jacobian $JF$ of $F$ is Lipschitz continuous with Lipschitz constant $L_{F}>0$ on $\Omega$ if 
\begin{equation}\label{lip}
	\| JF(x)-  JF(y)\|\leq L_{F}\|x-y\|
\end{equation}
for all $x,y\in\Omega$. In subsequent convergence analysis, we shall use the following condition:
\begin{equation*}
	{\rm (A)}\quad \exists\; L>0 \quad {\rm s.t.}\;\; \frac{L}{2}\|\cdot\|_{2}^{2}e-F(\cdot) \;{\rm is}\;C{\rm \text{-}convex\; on}\; \Omega.
\end{equation*}
Actually, the introduction of this condition is inspired by the Lipschitz-like (LC) condition proposed by Bauschke et al. in  \cite{BBT2017}. 

\begin{remark}\normalfont\label{re4.1}
	Let $m=1$, $C=\mathbb{R}_{+}$ and $e=1$. In such case, if we set $h=\|\cdot\|^{2}/2$ and $g=f_{1}$ in \cite{BBT2017}, then the condition (A) can be regarded as the LC condition. Although it was assumed that $g$ is convex in the LC condition \cite{BBT2017}, the convexity assumption of $g$ plays no role in the LC condition, and $g$ can be nonconvex, as already observed in \cite{B_f2018}. Therefore, our condition (A) does not require $C$-convexity of $F$.
\end{remark}

We will see that, in the following Lemma \ref{descent_lemma}, the mere translation of condition (A) into its first-order characterization immediately gives the new and key vector version of descent lemma. For convenience, we let 
$G(\cdot)=(L/2)\|\cdot\|_{2}^{2}e-F(\cdot)$ in the sequel.

\begin{lemma}\label{descent_lemma}
	The condition (A) is equivalent to
	\begin{equation}\label{descent_lemma0}
		F(y)-F(x)\preceq_{C} JF(x)(y-x)+\frac{L}{2}\|y-x\|_{2}^{2}e, \quad \forall x,y\in\Omega.
	\end{equation}
\end{lemma}
\begin{proof}
	For any $x,y\in \Omega$, the function $G(\cdot)$ is $C$-convex on $\Omega$ if and only if
	\begin{equation}\label{descent_lemma_0}
		JG(x)(y-x)\preceq_{C} G(y)-G(x).
	\end{equation}
	For $JG(x)(y-x)$ in (\ref{descent_lemma_0}), by a simple calculation, we have
	\begin{equation}\label{descent_lemma2}
		JG(x)(y-x)=L\langle x, y\rangle e-L\|x\|_{2}^{2}e-JF(x)(y-x).
	\end{equation}
	From the notion of $G(\cdot)$, (\ref{descent_lemma_0}) and  (\ref{descent_lemma2}), we obtain
	\begin{equation*}
		\begin{aligned}[b]
			F(y)-F(x)&\preceq_{C} JF(x)(y-x)+\frac{L}{2}\|y\|_{2}^{2}e-\frac{L}{2}\|x\|_{2}^{2}e-L\langle x, y\rangle e+L\|x\|_{2}^{2}e\\
			&=JF(x)(y-x)+\frac{L}{2}(\|y\|_{2}^{2}+\|x\|_{2}^{2}-2\langle x, y\rangle)e\\
			&=JF(x)(y-x)+\frac{L}{2}\|y-x\|_{2}^{2}e,
		\end{aligned}
	\end{equation*}
	and thus the proof is complete.
\end{proof}

For the condition (A) and Lemma \ref{descent_lemma}, we have the following some remarks.

\begin{remark}\normalfont\label{re4.2}
	Let $C=\mathbb{R}_{+}^{m}$ and $e=(1,1,\ldots,1)$. In this case, as reported in \cite{AFP2021}, by using the same idea in the proof \cite[Lemma 2.4.2]{D_n1996}, we can also obtain \eqref{descent_lemma0} with $L= L_F$ from the Lipschitz continuity of $JF$. Therefore, the Lipschitz continuity of $JF$ implies that the condition (A) in view of Lemma \ref{descent_lemma}. However, the following example shows that the reverse relationship does not necessarily hold.
\end{remark}

\begin{example}\normalfont\label{ex3.1}
	Let $\Omega=[0,1]$ and 
	$$F(x)=\left(1-\frac{2}{3}x^{3/2},\left(x-\frac{1}{2}\right)^{2}\right)^{\top},\quad \forall x\in\Omega.$$
	The image of $F$ is shown in Figure \ref{figure}(a). Consider $C=\mathbb{R}_{+}^{2}$. It is easy to check that for $L=2$, $G$ is $\mathbb{R}_{+}^{2}$-convex on $\Omega$ (see Figure \ref{figure}(b)). However, $JF$ is not Lipschitz continuous on $\Omega$. It should be carefully noted that the conditional gradient method with adaptative step size described in \cite{AFP2021} cannot be applicable to this problem since their method relies on the Lipschitz continuity of $JF$. Therefore, as we shall see, our method is not a simple formal improvement.
\end{example}

\begin{figure}[H]
	\centering  
	\subfigure[]{
		\label{pk1}
		\includegraphics[width=0.31\textwidth]{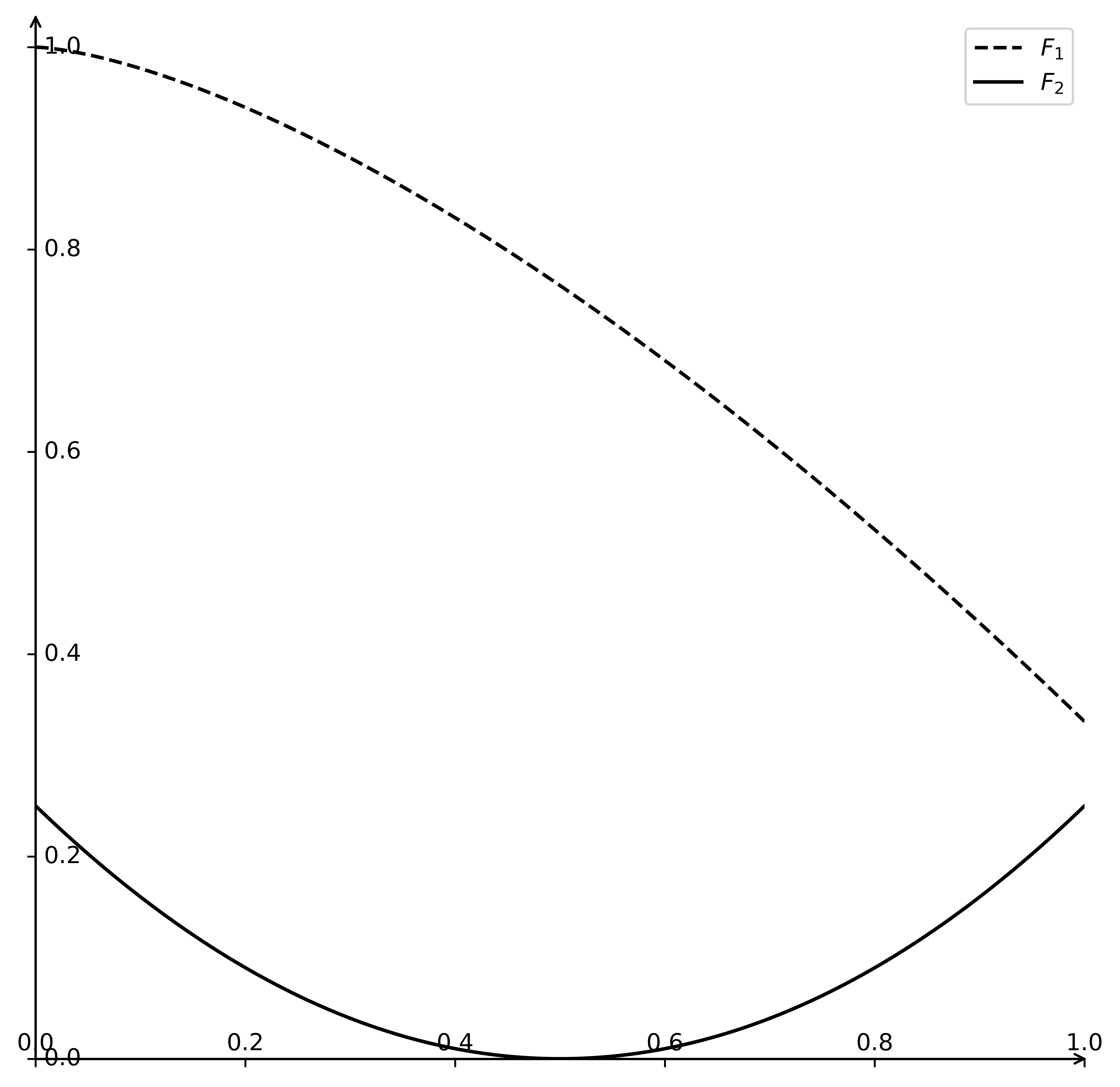}}
	\subfigure[]{
		\label{pk2}
		\includegraphics[width=0.31\textwidth]{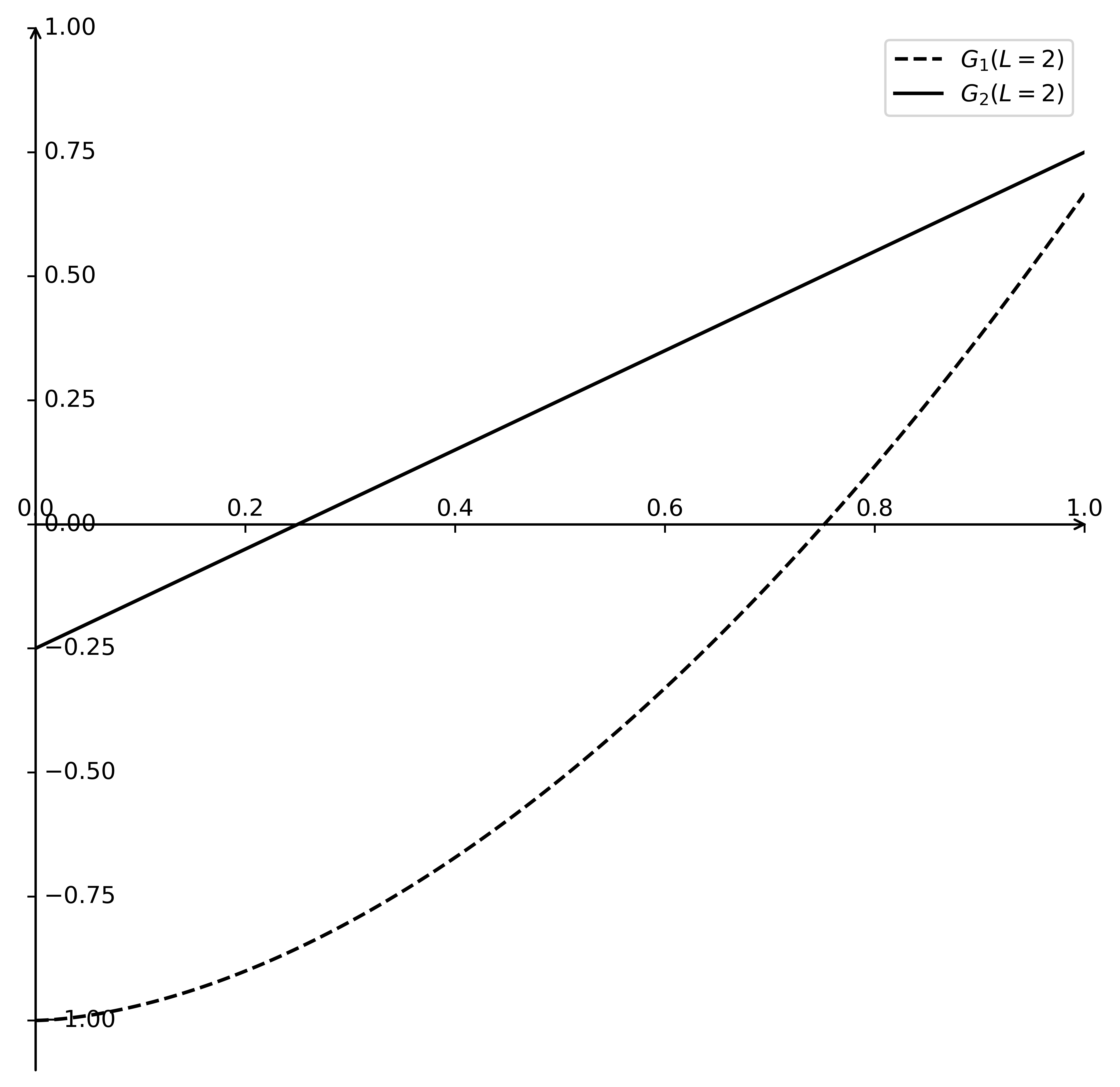}}
	\subfigure[]{
		\label{pk3}
		\includegraphics[width=0.31\textwidth]{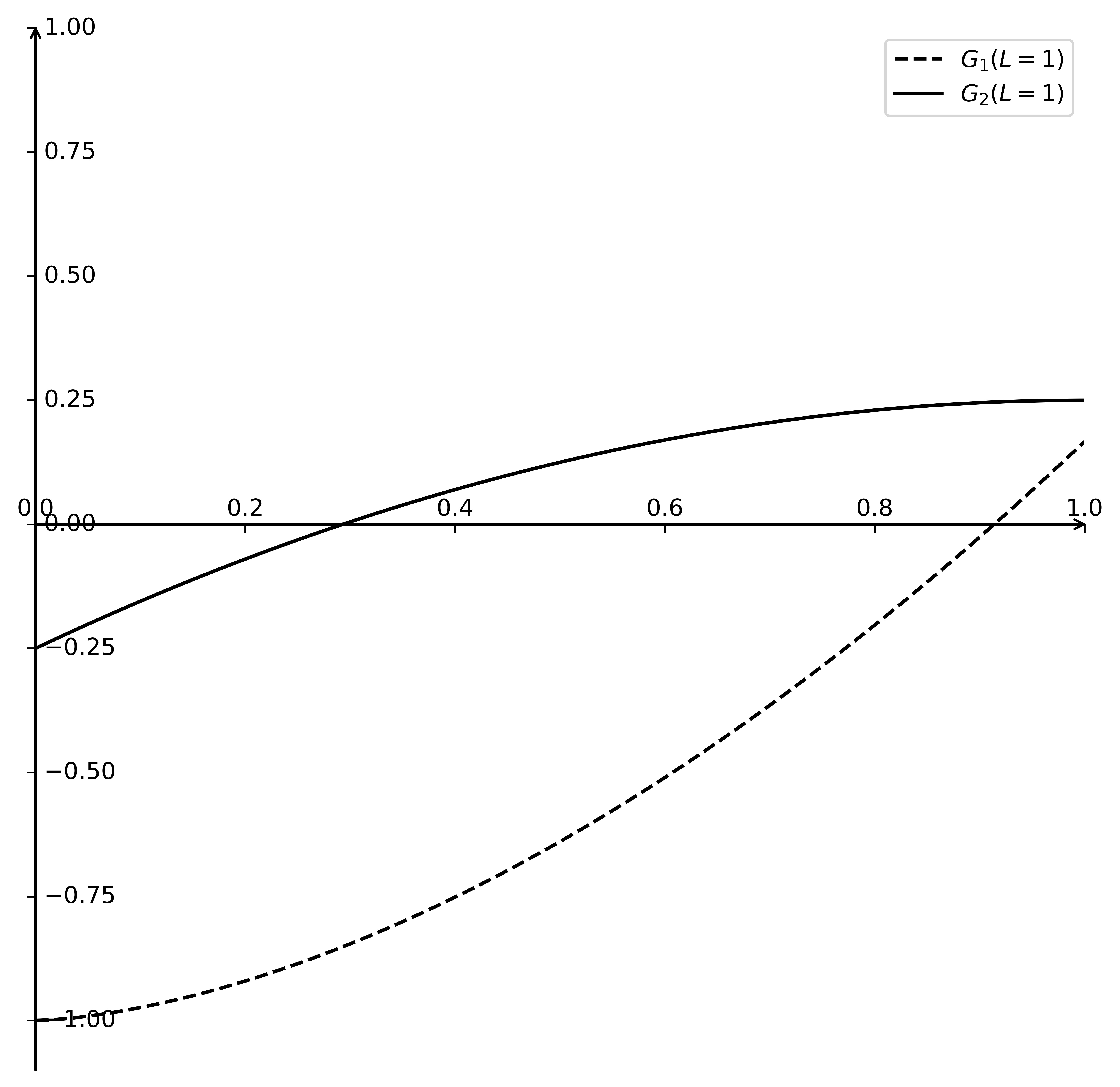}}
	\caption{ An illustration for $F$ and $G$.}
	\label{figure}
\end{figure}

\begin{remark}\normalfont\label{rem4.2}\label{rem4.3}
	For a general partial order $C$ rather than the non-negative orthant, the relationships between Lipschitz continuity of $JF$ and condition (A) seem less readily available. But we have the same observation as Remark \ref{re4.2} that the condition (A) does not imply the Lipschitz continuity of $JF$, the following example shows this case. Therefore, we propose an open question (Q1) in this paper: How to obtain the relations between them under some suitable assumptions for a general partial order $C$?
\end{remark}

\begin{example}\normalfont
	We continue to consider Example \ref{ex3.1}. Consider $C=\{y\in\mathbb{R}^{2}:0\leq y_{1}+y_{2},0\leq y_{1}\}$. It is easy to check that for $L=1$, $G$ is $C$-convex but not $\mathbb{R}_{+}^{2}$-convex on $\Omega$ (see Figure \ref{figure}(c)). However,  $JF$ is not Lipschitz continuous.
\end{example}

\begin{remark}\normalfont\label{descent_lemma1}
	Lemma \ref{descent_lemma} implies that, for any $x,y\in\Omega$,
	\begin{equation}\label{descent_lemma10}
		\varphi_{_{C}}(F(y))-\varphi_{_{C}}(F(x))\leq \psi_{x}(y)+\frac{L}{2}\|y-x\|_{2}^{2}\varphi_{_{C}}(e).
	\end{equation}
	Indeed, from (\ref{descent_lemma0}) and Lemma \ref{varphi_property}(iv)--(vi), we have
	\begin{equation}\label{descent_lemma11}
		\begin{aligned}
			\varphi_{_{C}}(F(y)-F(x))&\leq\varphi_{_{C}}\left(JF(x)(y-x)+\frac{L}{2}\|y-x\|_{2}^{2}e\right)\\
			&\leq\varphi_{_{C}}(JF(x)(y-x))+\varphi_{_{C}}\left(\frac{L}{2}\|y-x\|_{2}^{2}e\right)\\
			&=\varphi_{_{C}}(JF(x)(y-x))+\frac{L}{2}\|y-x\|_{2}^{2}\varphi_{_{C}}(e).
		\end{aligned}
	\end{equation}
	Obviously, according to Lemma \ref{varphi_property}(vi), it holds that
	\begin{equation}\label{descent_lemma12}
		\varphi_{_{C}}(F(y))-\varphi_{_{C}}(F(x))\leq\varphi_{_{C}}(F(y)-F(x)).
	\end{equation}
	Therefore, it immediately follows from (\ref{descent_lemma11}), (\ref{descent_lemma12}) and (\ref{psix}) that (\ref{descent_lemma10}) holds.
\end{remark}

We now give the adaptative step size strategy with the help of condition (A). 

\noindent\textbf{Adaptative step size} Assume that condition (A) holds. Define the step size as
\begin{equation}\label{ada}
	t_{k}=\min\left\{1,-\frac{v^{k}}{L\|d^{k}\|_{2}^{2}}\right\}
\end{equation}

Since $v(x)<0$ and $s(x)\neq x$ for nonstationary points, the adaptive step size for the CondG algorithm is well-defined. Before giving our convergence result, we first present an important property which is essential for showing convergence analysis.

\begin{lemma}\label{varphi_descent}
	Take $e\in {\rm int}(C)$ with $\varphi_{_{C}}(e)<2$. Suppose that condition (A) holds and $\{x^{k}\}$ is a sequence produced by the CondG algorithm with the adaptive step size. Then, for all $k$, it holds that
	\begin{equation}\label{d00}
		\varphi_{_{C}}(F(x^{k+1}))-\varphi_{_{C}}(F(x^{k}))\leq\frac{\varphi_{_{C}}(e)-2}{2}\min\left\{\frac{(v^{k})^{2}}{L({\rm diam}(\Omega))^{2}},-v^{k}\right\}.
	\end{equation}
\end{lemma}
\begin{proof}
	Let $x^{k+1}=x^{k}+t_{k}d^{k}$, where $d^{k}=s^{k}-x^{k}$ and
	\begin{equation}\label{d0}
		t_{k}=\min\left\{1,-\frac{v^{k}}{L\|d^{k}\|_{2}^{2}}\right\}.
	\end{equation}
	Since condition (A) holds, then by (\ref{descent_lemma10}) invoked with $x=x^{k}$ and $y=x^{k+1}$, we have
	\begin{equation}\label{d2}
		\begin{aligned}
			\varphi_{_{C}}(F(x^{k+1}))-\varphi_{_{C}}(F(x^{k}))&\leq\psi_{x^{k}}(x^{k}+t_{k}(s^{k}-x^{k}))+\frac{L}{2}t_{k}^{2}\|d^{k}\|_{2}^{2}\varphi_{_{C}}(e)\\
			&= t_{k}\psi_{x^{k}}(s^{k})+\frac{L}{2}t_{k}^{2}\|d^{k}\|_{2}^{2}\varphi_{_{C}}(e)\\
			&=t_{k}v^{k}+\frac{L}{2}t_{k}^{2}\|d^{k}\|_{2}^{2}\varphi_{_{C}}(e),
		\end{aligned}
	\end{equation}
	where the first equality holds in view of (\ref{psix}) and Lemma \ref{varphi_property}(iv). According to (\ref{d0}), there are two options:
	
	\emph{Case 1}. Let $t_{k}=1$. This, combined with (\ref{d0}), gives 
	\begin{equation}\label{d4}
		L\|d^{k}\|_{2}^{2}\leq -v^{k}.
	\end{equation}
	By (\ref{d2}) and (\ref{d4}), we obtain
	\begin{equation}\label{d5}
		\varphi_{_{C}}(F(x^{k+1}))-\varphi_{_{C}}(F(x^{k}))\leq\frac{2-\varphi_{_{C}}(e)}{2}v^{k}.
	\end{equation}
	
	\emph{Case 2}. Let $t_{k}=-v^{k}/(L\|d^{k}\|_{2}^{2})$. Clearly, $\|d^k\|=\|s^{k}-x^{k}\|\leq{\rm diam}(\Omega)$. This, together with (\ref{d2}) and the fact that $\varphi_{_{C}}(e)<2$, yields
	\begin{equation}\label{d3}
		\begin{aligned}
			\varphi_{_{C}}(F(x^{k+1}))-\varphi_{_{C}}(F(x^{k}))&\leq\frac{\varphi_{_{C}}(e)-2}{2}\frac{(v^{k})^{2}}{L\|d^{k}\|_{2}^{2}}\\
			&\leq\frac{\varphi_{_{C}}(e)-2}{2}\frac{(v^{k})^{2}}{L({\rm diam}(\Omega))^{2}}.
		\end{aligned}
	\end{equation}
	Therefore, (\ref{d00}) is directly derived by (\ref{d5}) and (\ref{d3}).
\end{proof}

\begin{theorem}\label{vk0}
	Suppose that all conditions of Lemma \ref{varphi_descent} hold. Then, every accumulation point of $\{x^{k}\}$ is a stationary point of problem (\ref{mop}).  
\end{theorem}
\begin{proof}
	From $(\ref{d00})$ and $\varphi_{_{C}}(e)<2$ and observing that $v^{k}<0$, we have for all $k$,
	\begin{equation}\label{vk00}
		\varphi_{_{C}}(F(x^{k+1}))-\varphi_{_{C}}(F(x^{k}))\leq\frac{\varphi_{_{C}}(e)-2}{2}\min\left\{\frac{(v^{k})^{2}}{L({\rm diam}(\Omega))^{2}},-v^{k}\right\}\\
		<0,
	\end{equation}
	i.e., $\varphi_{_{C}}(F(x^{k+1}))<\varphi_{_{C}}(F(x^{k}))$, which implies that $\{\varphi_{_{C}}(F(x^{k}))\}$ is nonincreasing for all $k$. By Remark \ref{feasible} and the continuity of $F$, there exits $\bar{F}\in\mathbb{R}^{m}$ such that $\bar{F}\preceq_{C} F(x^{k})$ for all $k$, it follows from Lemma \ref{varphi_property}(vi) that $\varphi_{_{C}}(\bar{F})\leq\varphi_{_{C}}(F(x^{k}))$ for all $k$. Therefore, we conclude that the sequence $\{\varphi_{_{C}}(F(x^{k}))\}$ is convergent. This obviously means that 
	\begin{equation}\label{vk001}
		\lim_{k\rightarrow\infty}(\varphi_{_{C}}(F(x^{k+1}))-\varphi_{_{C}}(F(x^{k})))=0.
	\end{equation}
	Taking $\lim_{k\rightarrow\infty}$ in (\ref{vk00}), and then combining with (\ref{vk001}), we have
	\begin{equation}\label{vk01}
		\lim_{k\rightarrow\infty}v(x^k)=0.
	\end{equation} 
	From Proposition \ref{pa_sta_equ}(ii), Proposition \ref{opt_val_continuous} and (\ref{vk01}), we obtain that each accumulation point of $\{x^{k}\}$ is a stationary point of problem (\ref{mop}). 
\end{proof}

\begin{remark}\normalfont
	If we consider the norm $\|\cdot\|_{\infty}$, $C=\mathbb{R}_{+}^{m}$ and $e=(1,1,\ldots,1)\in{\rm int}(\mathbb{R}_{+}^{m})$, then $\varphi_{_{C}}(e)=1<2$, and thus Theorem \ref{vk0} is an improvement of the corresponding convergence result presented in \cite{AFP2021}. 
\end{remark}

According to Theorems \ref{stationary_weakly_pareto} and \ref{vk0}, we have the following result.

\begin{theorem}\normalfont
	If $F$ is $C$-convex on $\Omega$ and all conditions of Theorem \ref{vk0}, then the sequence $\{x^{k}\}$ produced by the CondG algorithm with the adaptative step size strategy \eqref{ada} converges to a weakly efficient solution of problem (\ref{mop}).
\end{theorem}

\subsection{Analysis of CondG algorithm with nonmonotone line search}

In this section, we extend a nonmonotone line search for scalar optimization proposed in \cite{G_i2008} to the vector setting and establish the convergence  property of the CondG algorithm with it. 

In the sequel, we denote $$D^{k}=(D_{1}^{k},D_{2}^{k},\ldots,D_{m}^{k})^{\top}\in\mathbb{R}^{m}.$$

\noindent\textbf{Nonmonotone line search} Let $\beta\in(0,1)$, $\delta\in(0,1)$, $\eta_{\max}\in[0,1)$ and $\tau\in(0,1]$. Choose $t_{k}$ as the largest value in $\{\tau,\delta\tau,\delta^{2}\tau,\ldots\}$ such that
\begin{equation}\label{nonmonotone_line_search}
	F(x^{k}+t_{k}d^{k})\preceq_{C} D^{k}+\beta t_{k} JF(x^{k})d^{k},
\end{equation}
where
\begin{equation}\label{nonmonotone_line_search1}
	\left\{\begin{array}{lllll}
		D^{0}=F(x^{0}) &\quad k=0,\\
		D^{k}=\eta_{k} D^{k-1}+(1-\eta_{k})F(x^{k}), &\quad k\geq1,
	\end{array}\right.
\end{equation}
with $0\leq\eta_{k}\leq\eta_{\max}<1$.

\begin{remark}\normalfont
	The current nonmonotone term in \eqref{nonmonotone_line_search1} is a convex combination of the previous nonmonotone term and the current objective function value, instead of the maximum of recent objective function values or an average of the successive objective function values. Note that $\eta_{k}$ controls the degree of nonmonotonicity in $D^{k}$. If $\eta_{k}= 0$ for all $k$, then $D^{k}=F(x^{k})$ and, as a consequence, the nonmonotone line search reduces to the Armijio line search \eqref{armijio}.
\end{remark}

\begin{lemma}\label{fk_leq_dk}
	Let $\{x^{k}\}$ be the sequence generated by the CondG algorithm with the nonmonotone line search. If $JF(x^{k})d^{k}\preceq_{C} 0$, then $F(x^{k})\preceq_{C} D^{k}$ for all $k\geq0$.
\end{lemma}

\begin{proof}
	We proceed by induction. Obviously, $F(x^{k})= D^{k}$ for $k=0$. Assume that the conclusion holds for $k>0$. We shall prove that $F(x^{k+1})\preceq_{C} D^{k+1}$ for $k+1$. By $JF(x^{k})d^{k}\preceq_{C}0$ and \eqref{nonmonotone_line_search}, one has
	\begin{equation}\label{fk_leq_dk2}
		F(x^{k+1})\preceq_{C} D^{k}.
	\end{equation}
	By \eqref{nonmonotone_line_search1}, we have $D^{k+1}=\eta_{k+1} D^{k}+(1-\eta_{k+1})F(x^{k+1})$. Re-arranging and subtracting $F(x^{k+1})$ from both sides of this equality gives
	\begin{equation*}
		D^{k+1}-F(x^{k+1})=\eta_{k+1}(D^{k}-F(x^{k+1})),
	\end{equation*}
	which, combined with (\ref{fk_leq_dk2}), yields that $F(x^{k+1})\preceq_{C} D^{k+1}$. This completes the proof.
\end{proof}

In the next proposition, we prove that the CondG algorithm with the nonmonotone line search is well-defined, so that the iterates $\{x^{k}\}$ can be generated.

\begin{proposition}\label{non_des_lemma}
	Let $x^{k}$ be an iterate of the CondG algorithm with the nonmonotone line search. If $JF(x^{k})d^{k}\prec_{C}0$, then there exists some $\hat{t}\in(0,1]$ such that
	\begin{equation*}
		F(x^{k}+t_{k}d^{k})\prec_{C} D^{k}+\beta t JF(x^{k})d^{k},\quad \forall t\in(0,\hat{t}].
	\end{equation*}
\end{proposition}

\begin{proof}
	The proof follows from \cite[Proposition 2.1]{DS2005} and Lemma \ref{fk_leq_dk}.
\end{proof}

\begin{lemma}\label{non_convergence}
	Let $\{x^{k}\}$ be the sequence generated by the CondG algorithm with the nonmonotone line search. Then, the following statements hold:
	\begin{enumerate}[{\rm(i)}]
		\item The sequence $\{D^{k}\}$ is monotone decreasing;
		\item For all $k\geq0$, $x^{k}\in\mathcal{L}\cap\Omega$, where $\mathcal{L}=\{x\in\mathbb{R}^{n}:F(x)\preceq_{C} F(x^{0})\}$ and $x^{0}\in\mathbb{R}^{n}$ is a given initial point.
	\end{enumerate}
\end{lemma}

\begin{proof}
	Let us first prove item (i). According to \eqref{nonmonotone_line_search1} and \eqref{fk_leq_dk2}, for all $k\geq0$, we have
	\begin{equation}\label{non_convergence1}
		D^{k+1}=\eta_{k+1} D^{k}+(1-\eta_{k+1})F(x^{k+1})\preceq_{C}\eta_{k+1} D^{k}+(1-\eta_{k+1})D^{k}=D^{k},
	\end{equation}
	which implies that item (i) holds.
	
	We proceed item (ii) by induction. For $k=0$, $x^{0}\in\mathcal{L}$ is clear. Assume that $x^{k}\in\mathcal{L}$ when $k\leq l$, that is, $F(x^{k})\preceq_{C} F(x^{0})$ for all $k\leq l$. Note that, by Lemma \ref{fk_leq_dk} and \eqref{non_convergence1}, one has
	\begin{equation*}
		F(x^{l+1})\preceq_{C} D^{l+1}\preceq_{C} D^{l}\preceq_{C}\cdots\preceq_{C} D^{0}=F(x^{0}),
	\end{equation*}
	i.e., $x^{k+1}\in\mathcal{L}$. This, together with Remark \ref{feasible}, gives $x^{k}\in\mathcal{L}\cap\Omega$.  
\end{proof}

\begin{lemma}\label{Dk_convergence}
	The sequence $\{D^{k}\}$ is convergent, and
	\begin{equation}\label{Dk_convergence1}
		\lim_{k\rightarrow\infty}\varphi_{_{C}}(F(x^{k+1})-D^{k})= 0.
	\end{equation}
\end{lemma}

\begin{proof}
	Since $F$ is continuous, by Lemma \ref{non_convergence}(ii) and Lemma \ref{fk_leq_dk}, there exist $\bar{F}\in\mathbb{R}^{m}$ such that $\bar{F}\preceq_{C} F(x^{k})\preceq_{C} D^{k}$, i.e., $\{D^{k}\}$ is bounded from below. Hence, it is convergent by Lemma \ref{non_convergence}(i). From the continuity of $\varphi_{_{C}}$, we immediately get $\{\varphi_{_{C}}(D^{k})\}$ is convergent. Using \eqref{nonmonotone_line_search1} and \eqref{fk_leq_dk2}, we have
	\begin{equation}\label{Dk_convergence2}
		D^{k+1}-D^{k}=(1-\eta_{k+1})(F(x^{k+1})-D^{k})\preceq_{C}(1-\eta_{\max})(F(x^{k+1})-D^{k})\preceq_{C}0.	
	\end{equation}
	By \eqref{Dk_convergence2} and Lemma \ref{varphi_property}(iv)--(vi) and observing that $\varphi_{_{C}}(0)=0$, we obtain
	\begin{equation*}
		\varphi_{_{C}}(D^{k+1})-\varphi_{_{C}}(D^{k})\leq \varphi_{_{C}}(D^{k+1}-D^{k})\leq(1-\eta_{\max})\varphi_{_{C}}(F(x^{k+1})-D^{k})\leq0.
	\end{equation*}
	Therefore, \eqref{Dk_convergence1} holds by taking $\lim_{k\rightarrow\infty}$ on the above relation.
\end{proof}

\begin{theorem}\label{non_convergence_analysis}
	Let $\{x^{k}\}$ be a sequence produced by the CondG algorithm with the nonmonotone line search \eqref{nonmonotone_line_search}. Then, every accumulation point of $\{x^{k}\}$ is a stationary point of problem (\ref{mop}).
\end{theorem}

\begin{proof}
	Assume that $\hat{x}\in\Omega$ is a accumulation point of the sequence $\{x^{k}\}$, i.e., there exists a subsequence $\{x_{k_{j}}\}$ of $\{x^{k}\}$ such that $\lim_{j\rightarrow\infty}x_{k_{j}}=\hat{x}$. According to Proposition \ref{pa_sta_equ}(ii), it is enough to see that $v(\hat{x})=0$. By \eqref{nonmonotone_line_search} and Lemma \ref{varphi_property}(iv)--(vi), we have
	\begin{equation}\label{non_convergence_analysis1}
		\varphi_{_{C}}(F(x^{k}+t_{k}d^{k})-D^{k})\leq\beta t_{k}\varphi_{_{C}}(JF(x^{k})d^{k}).
	\end{equation}
	Using \eqref{Dk_convergence1}, \eqref{non_convergence_analysis1} and the fact that $JF(x^{k})d^{k}\prec_{C}0$, as well as Lemma \ref{varphi_property}(ii), we get
	\begin{equation*}
		0=\lim_{k\rightarrow\infty}\varphi_{_{C}}(F(x^{k}+t_{k}d^{k})-D^{k})\leq\lim_{k\rightarrow\infty}\beta t_{k}\varphi_{_{C}}(JF(x^{k})d^{k})\leq0.
	\end{equation*}
	Therefore, 
	\begin{equation}\label{non_convergence_analysis2}
		\lim_{k\rightarrow\infty} t_{k}v(x^{k})=\lim_{k\rightarrow\infty} t_{k}\varphi_{_{C}}(JF(x^{k})d^{k})=0.
	\end{equation}
	We now consider the subsequence $\{t_{k_{j}}\}$ of $\{t_{k}\}$. Taking into consideration that $t_{k}\in(0,1]$ for all $k$, we have the following two cases
	\begin{equation}\label{non_convergence_analysis3}
		{\rm (a)}\;\mathop{\lim\sup}_{j\rightarrow\infty} t_{k_{j}}>0\qquad{\rm or}\qquad{\rm (b)}\;\mathop{\lim\sup}_{j\rightarrow\infty} t_{k_{j}}=0.
	\end{equation}
	Repeating now the corresponding arguments in the proof of Theorem \ref{convergence}, we arrive at \eqref{convergence4_0}, i.e.,
	\begin{equation}\label{non_convergence_analysis3_1}
		v(\hat{x})\leq0.
	\end{equation}
	Take now some fixed but arbitrary positive integer $l$. Since $\lim_{i\rightarrow\infty}t_{k_{j_{i}}}=0$, we have $t_{k_{j_{i}}}<\tau^{l}$ for $i$ large enough, which means that the nonmonotone line search condition does not hold at $x^{k_{j_{i}}}$ for $t=\tau^{l}$, i.e.,
	$$F(x^{k_{j_{i}}}+\tau^{l}d(x^{k_{j_{i}}}))\npreceq_{C} D^{k_{j_{i}}}+\tau^{l}\beta JF(x^{k_{j_{i}}})d(x^{k_{j_{i}}}).$$
	Equivalently,
	$$F(x^{k_{j_{i}}}+\tau^{l}d(x^{k_{j_{i}}}))-D^{k_{j_{i}}}-\tau^{l}\beta JF(x^{k_{j_{i}}})d(x^{k_{j_{i}}})\notin-C,$$
	which means that
	\begin{equation}\label{non_convergence_analysis4}
		F(x^{k_{j_{i}}}+\tau^{l}d(x^{k_{j_{i}}}))-D^{k_{j_{i}}}-\tau^{l}\beta JF(x^{k_{j_{i}}})d(x^{k_{j_{i}}})\in\mathbb{R}^{m}\backslash(-C)={\rm int}(\mathbb{R}^{m}\backslash(-C)),
	\end{equation}
	By (\ref{non_convergence_analysis4}) and Lemma \ref{varphi_property}(ii), we have
	\begin{equation}\label{non_convergence_analysis5}
		\varphi_{_{C}}(F(x^{k_{j_{i}}}+\tau^{l}d(x^{k_{j_{i}}}))-D^{k_{j_{i}}}-\tau^{l}\beta JF(x^{k_{j_{i}}})d(x^{k_{j_{i}}}))>0.
	\end{equation}
	Since \eqref{non_convergence_analysis5} holds for any positive integer $l$, using Proposition \ref{non_des_lemma} with $k=k_{j}$ and Lemma \ref{varphi_property}(vi), we conclude that $JF(x^{k_{j}})d(x^{k_{j}})\nprec_{C}0$, i.e., $JF(x^{k_{j}})d(x^{k_{j}})\in\mathbb{R}^{m}\backslash(-{\rm int}(C))$. By Lemma \ref{varphi_property}(ii), it holds that 
	\begin{equation}\label{non_convergence_analysis6}
		\varphi_{_{C}}(JF(x^{k_{j}})d(x^{k_{j}}))\geq0.
	\end{equation}
	Since $F$ is continuously differentiable and $\varphi_{_{C}}$ is continuous, taking limits in (\ref{non_convergence_analysis6}) with $i\rightarrow\infty$, we have $v(\hat{x})=\varphi_{_{C}}(JF(\hat{x})d(\hat{x}))\geq0$. This, combined with (\ref{non_convergence_analysis3_1}), yields that $v(\hat{x})=0$.
\end{proof}

From Theorems \ref{stationary_weakly_pareto} and \ref{non_convergence_analysis}, we obtain the following result.

\begin{theorem}\normalfont
	If $F$ is $C$-convex on $\Omega$ and all conditions of Theorem \ref{non_convergence_analysis}, then the sequence $\{x^{k}\}$ produced by the CondG algorithm with  the nonmonotone line search \eqref{nonmonotone_line_search} converges to a weakly efficient solution of problem (\ref{mop}).
\end{theorem}

\section{Conclusion}\label{sec13}

In this work, we have proposed a conditional gradient method for solving constrained vector optimization problems. In our method, the auxiliary subproblem used to obtain the descent direction is constructed based on the well-known Hiriart-Urruty's oriented distance function, and the step sizes are obtained by Armijio, adaptative and nonmonotone strategies. We proved the sequence generated by the method can converge to a stationary point no matter how bad is our initializing point. It is worth mentioning that in the convergence analysis of the conditional gradient method with adaptative step size rule, we do not assume the Lipschitz continuity of the objective's Jacobian, but with the help of a flexible convexity condition given by us. Such setting is obviously different from the work in \cite{AFP2021}.

An open question (Q1) has been proposed naturally in Remark \ref{rem4.3}, and it is worth investigating in our future work. Furthermore, it would be interesting in the future to revisit the convergence results of the existing methods for multiobjective/vector optimization problems by using our proposed condition (A) instead of the Lipschitz continuity of Jacobian.

It is noteworthy that our method is conceptual and theoretical schemes rather than implementable algorithms. Similar issues also appear in the literature; see, e.g., \cite{DS2005, FD2013, LC2014, QJJZ2017, C2010, CY2012, C2013, BG2018, GP2020, DI2004}). Therefore, the computational efficiency of the method to a real-world optimization problem depends essentially on the choice of a good feature and structure of the minimization subproblem \eqref{sca_pro1} at every iteration. For example, when the norm $\|\cdot\|_{\infty}$ is used and $C=\mathbb{R}_{+}^{m}$, the objective function $\psi_{x}$ in (\ref{sca_pro1}) has the simple form \eqref{sca_pro11}, and then the descent direction can be easily computed by program. In such situation, it is worthwhile to verify the performance of our method with the nonmonotone line search. For the general case, as recommended in \cite{CY2012}, we can use the bundle method presented in \cite{H_c1993} to solve the problem \eqref{sca_pro1}. We leave these issues as subjects for future researches.

\end{document}